\documentclass[reqno]{amsart}
\usepackage[english]{babel}
\usepackage{amscd,amssymb,amsmath,amsfonts,latexsym,amsthm}
\usepackage{inputenc}
\usepackage{graphicx}
\newtheorem{thm}{Theorem}[section]

\newtheorem{prop}[thm]{Proposition}
\newtheorem{defn}[thm]{Definition}
\newtheorem{exam}[thm]{Example}

\numberwithin{equation}{section}

\begin{document}


\title[Few remarks on evolution algebras]
{Few remarks on evolution algebras}%
\author{Abror Kh. Khudoyberdiyev,\  Bakhrom A. Omirov, \ Izzat Qaralleh}

\address{[A.\ Kh.\ Khudoyberdiyev and B.\ A.\ Omirov] Institute of Mathematics, National University of Uzbekistan,
Tashkent, 100125, Uzbekistan.} \email{khabror@mail.ru, omirovb@mail.ru}\

\address{[Izzat Qaralleh] I
Department of Computational \& Theoretical Sciences, Faculty of
Science, International Islamic University Malaysia, P.O. Box, 141,
25710, Kuantan, Pahang, Malaysia} \email{izzat\_math@yahoo.com}\

%

\begin{abstract}
In the present paper we study some algebraic properties of evolution
algebras. Moreover, we reduce the study of evolution algebras of
permutations to two special types of evolution algebras, idempotents
and absolute nilpotent elements of the algebra. We study
three-dimensional evolution algebras whose each element of evolution
basis has infinite period. In addition, for an evolution algebra
with some properties we describe its associative enveloping
algebra.\\[2mm]
\end{abstract}

\maketitle
\section{Introduction.}

In 20s and 30s of the last century the new object was introduced to mathematics, which was the product of
interactions between Mendelian genetics and mathematics. One of the first scientist who give an algebraic
interpretation of the $"\times"$ sign, which indicated sexual reproduction was Serebrowsky.

It is known that there exists an intrinsic and general mathematical structure behind the neutral Wright-Fisher models in population genetics, the reproduction of bacteria involved by bacteriophages, asexual reproduction or generally non-Mendelian inheritance and Markov chains. In \cite{Tian} a new type of algebras was
associated with it -- the evolution algebras.

Although an evolution algebra is an abstract system, it gives an insight for the study of non-Mendelian genetics. For instance, an evolution algebra can be applied to the inheritance of organelle genes, one can predict, in particular, all possible mechanisms to establish the homoplasmy of cell populations.

The general genetic algebras  developed into a field of independent
mathematical interest, because these algebras are in general
non-associative and do not belong to any of the well-known classes
of non-associative algebras such as Lie algebras, alternative
algebras, or Jordan algebras.

Until 1980s, the most comprehensive reference in this area was
W\"orz-Busekros's book \cite{Busekros}. More recent results, such as
genetic evolution in genetic algebras, can be found in Lyubich's
book \cite{ly}. A good survey is Reed's article \cite{Reed}. In
Tian's book \cite{Tian} a foundation of the framework of the theory
of evolution algebras is established and some applications of
evolution algebras in the theory of stochastic processes and
genetics are discussed. Recently, Rozikov and Tian \cite{Rozikov}
studied algebraic structures evolution algebras associated with
Gibbs measures defined on some graphs. In \cite{Camacho2},
\cite{Casas3}, \cite{Ladra} derivations, some properties of chain of
evolution algebras and dibaricity of evolution algebras were
studied. In \cite{Camacho1}, \cite{Casas1}, \cite{Casas2} certain
algebraic properties of evolution algebras (like right nilpotency,
nilpotency and solvability etc.) in terms of matrix of structural
constants have been investigated. In fact, nilpotency, right
nilpotency and solvability might be interpreted in a biological way
as a various types of vanishing ("deaths") populations.

The present paper is organized as follows: In Section 2 we give some
definitions and preliminary results. In Section 3 we reduce the
study of arbitrary evolution algebra of permutations into two
special evolution algebras. Section 4 is devoted to the description
of $n$-dimensional associative enveloping algebras of
$n$-dimensional evolution algebras with some restrictions on $rank$
of the matrix $A$ of structural constants. Moreover, associative
enveloping algebras for $2$-dimensional evolution algebras are
described, as well. In Section 5 we establish some properties of
three-dimensional evolution algebras whose each basis element has
infinite period.

Throughout the paper we consider finite-dimensional evolution
algebras over a field of zero characteristic. Moreover, in the
multiplication table of an evolution algebra the omitted products
are assumed to be zero.

\section{Preliminaries.}

Let us define the main object of this work - evolution algebra.

\begin{defn} \cite{Tian} Let $(E, \cdot)$ be an algebra over a field $F.$ If it admits a
basis $\{e_1, e_2, \dots\}$ such that $$e_i \cdot e_j = 0, \quad for \
i \neq j,\qquad  e_i \cdot e_i = \sum\limits_ka_{i,k}e_k, \quad
for  \ any \ i,$$ then this algebra is called evolution algebra.
\end{defn}
It is remarkable that this type of algebra depends on evolution basis $\{e_1, e_2, \dots\}.$

In the following theorem we present the list (up to isomorphism) of 2-dimensional complex evolution algebras.
\begin{thm} \cite{Casas2} \label{thm22}
Any  2-dimensional non abelian complex evolution algebra $E$ is isomorphic to
one of the following pairwise non isomorphic algebras:
\begin{enumerate}
\item $\dim E^2=1$
\begin{itemize}
\item $E_1:\ \ e_{1}e_{1} = e_{1}$,

\item $E_2: \ \ e_{1}e_{1} = e_{1}, \ \ e_{2}e_{2}= e_{1}$,

\item $E_3: \ \ e_{1}e_{1} = e_{1} + e_{2}, \ \  e_{2}e_{2}= -e_{1}-
e_{2}$,

\item $E_4: \ \ e_{1}e_{1} = e_{2}$.
\end{itemize}
\item $\dim E^{2}=2$
\begin{itemize}

\item $E_5: \ \  e_{1}e_{1}=e_{1}+a_{2}e_{2}, \ \
e_{2}e_{2}=a_{3}e_{1}+e_{2}, \ \ 1 - a_{2}a_{3}\ne 0$, where
$E_5(a_{2},a_{3})\cong E_5'(a_{3},a_{2})$,

\item $E_6: \ \ e_{1}e_{1}=e_{2}, \ \ e_{2}e_{2}=e_{1}+a_{4}e_{2}$,
where for $a_4\ne 0$,  $E_6(a_{4})\cong E_6(a'_{4}) \
\Leftrightarrow \ \frac{a'_4}{a_4}=\cos\frac{2\pi k}{3} + i
\sin\frac{2\pi k}{3} \ \mbox{for some} \ k=0, 1, 2$.
\end{itemize}
\end{enumerate}
\end{thm}

Further we shall show the role of idempotents and absolute nilpotent elements of an evolution algebra.

\begin{defn} An element $x$ of an evolution algebra $E$ is called idempotent, if $xx=x$.
An element $y$ of an evolution algebra $E$ is called absolute
nilpotent if $yy=0$.
\end{defn}

Consider a complex evolution algebra $E_{n,\pi}(a_1, a_2, \dots,
a_n)$ with a basis $\{e_1, e_2, \dots ,e_n\}$ and the table of multiplications given
by $$\left\{\begin{array}{ll}e_i\cdot e_i = a_ie_{\pi(i)}, &\ 1\leq i \leq n,\\[1mm]
e_i\cdot e_j = 0, & \ i \neq j,
\end{array}\right.$$
where $\pi$ is an element of the group of permutations $S_n$.

An evolution algebra $E_{n,\pi}(a_1, a_2, \dots,
a_n)$ is said to be {\it evolution algebra of permutations}.

In what follows, by a {\it cycle permutation} we mean a permutation
in which a part of symbols $\{l_1, l_2, \dots, l_t\}\subseteq \{1,
2, \dots, n\}$ are cyclic permutated and the rest ones are
stationary, i.e., $l_1 \rightarrow  l_2  \rightarrow \dots
\rightarrow  l_t \rightarrow  l_1$, and we denote $\pi=(l_1, l_2,
\dots, l_t).$

It is known that any permutation is up to order uniquely decomposed into product of independent cycles.

For permutations of the form $\pi = (l_1, l_2, \dots, l_r)(m_1, m_2, \dots, m_s) \dots (p_1, p_2, \dots, p_t)$ it is known the following result.
\begin{prop} Two permutations are conjugated in $S_n$ if and only if the corresponding sets
$\{r, s, \dots, t\}$ are coincided.
\end{prop}

\begin{defn} \cite{Tian} An evolution algebra $E$ with a table of multiplications $$e_i \cdot e_i =
\sum\limits_ka_{i,k}e_k, \quad a_ie_j=0, \ i\neq j$$ is called
Markov evolution algebra if $\sum\limits_ka_{i,k}=1.$
\end{defn}

For a given element $x$ of an evolution algebra $E$ we consider the
right multiplication operator $R_x \ : \ E \rightarrow E$ defined by
$R_x(y)=yx, \ y\in E.$

Note that operators of right and left multiplications are coincided,
since evolution algebras are commutative.

For an evolution algebra $E$, by $M(E)$ we denote an {\it
associative enveloping algebra} which is generated by the set
$R(E)=\{R_x \ | \ x\in E\}$. It is clear that $M(E)$ is a
subalgebra of $End(E).$

For an element $x \in E$ we define plenary powers as follows:
$$x^{[1]}=x, \quad x^{[k+1]} = x^{[k]} \cdot x^{[k]}, \quad k \geq 1.$$

\begin{defn} Let $e_j$ be a generator of an evolution algebra $E,$ the period $d$ of $e_j$ is defined
to be the greatest common divisor of the set $\{\log_2m \ | \ e_j < e_j^{[m]}\}.$ That is
$$d = g.c.d. \{\log_2m \ | \ e_j < e_j^{[m]}\}.
$$
\end{defn}


\section{Evolution algebra of permutations}

Let us first present two important examples of evolution algebra of
permutations.

\begin{exam} Consider the following evolution algebra:
$$E_{n}: \quad  \left\{\begin{array}{ll}e_i\cdot e_i = e_{i+1}, & 1 \leq i \leq n-1,\\[1mm]
e_n\cdot e_n = e_{1}, & \\[1mm]
e_i\cdot e_j = 0, & \ i \neq j.
\end{array}\right.$$
Evidently, the algebra $E_{n}$ is evolution algebra of permutations of the form $E_{n,\pi}(1, 1, \dots,1),$ with $\pi= (1, 2, 3, \dots, n).$
\end{exam}

\begin{exam} Evolution algebra defined as follows:
$$EN_{n}: \quad \left\{\begin{array}{ll}e_i\cdot e_i = e_{i+1}, & 1 \leq i \leq n-1,\\[1mm]
e_n\cdot e_n = 0, & \\[1mm]
e_i\cdot e_j = 0, & i \neq j,
\end{array}\right.$$
is the algebra of permutations of the form $E_{n,\pi}(1, 1, \dots,1, 0)$
with $\pi= (1, 2, 3, \dots, n).$
\end{exam}

Note that $E_n,$ $EN_n$ are single-generated simple and nilpotent
evolution algebras, respectively. Moreover, algebras $E_1,$ $EN_1$
define one-dimensional evolution algebras, whose basis elements
are idempotent and absolute nilpotent elements, respectively.

Now we shall consider some properties of evolution algebra of permutations.

\begin{prop} \label{prop33} Let $E_{n,\pi}(a_1, a_2, \dots, a_n)$ be an evolution algebra of permutations with the following conditions:

(i) $a_i\neq 0$ for all $i \ (1\leq i \leq n),$

(ii) $\pi=\pi_1 \circ\pi_2 \circ\dots \circ\pi_r,$ where $\pi_1 = (l_1, l_2, \dots, l_{k_1}), \ \ \pi_2 = (m_1, m_2, \dots,
m_{k_2}), \ \ \dots, \ \ \pi_r = (p_1, p_2, \dots, p_{k_r})$ are independent cycles and $k_1+k_2+\dots+k_r=n.$

Then $$E_{n,\pi}(a_1, a_2, \dots, a_n)\cong E_{k_1,\pi_1}(b_1, b_2, \dots, b_{k_1})\oplus E_{k_2,\pi_2}(c_1, c_2, \dots, c_{k_2})\oplus \dots \oplus E_{k_r, \pi_{r}}(d_1, d_2, \dots, d_{k_r}).$$
\end{prop}
\begin{proof}

The isomorphism is provided by the following change of basis:

$$e_{i,1}=e_{l_i}, \ 1\leq i \leq k_1, \quad e_{i,2}=e_{m_i}, \ 1\leq i \leq k_2, \quad \dots, \quad e_{i,r}=e_{p_i}, \ 1\leq i \leq k_r.$$

Thus, we have the evolution algebra $E_{k_s,\pi_s}(*, *, \dots, *)$ with the basis $e_{i,s}, \ 1\leq i \leq k_s, \ 1\leq s \leq r$  and $$E_{n,\pi}(a_1, a_2, \dots, a_n)\cong E_{k_1,\pi_1}(b_1, b_2, \dots, b_{k_1})\oplus E_{k_2,\pi_2}(c_1, c_2, \dots, c_{k_2})\oplus \dots \oplus E_{k_r, \pi_{r}}(d_1, d_2, \dots, d_{k_r})$$ for some non-zero  values of $b_i, c_j, \dots, d_s.$
\end{proof}

In the following proposition we specify more details on the terms of direct sum in the statement of Proposition \ref{prop33}.

\begin{prop}\label{prop34}
Any evolution algebra of permutation $E_{n,\tau}(a_1, a_2, \dots, a_n)$ with $\tau=(l_1, l_2, \dots, l_n)$ and condition $a_i\neq 0$ for all $i \ (1\leq i \leq n)$ is isomorphic to the algebra $E_{n,\pi}(a_1, a_{\pi(1)}, \dots, a_{\pi^{n-1}(1)})$ with $\pi=(1, 2, \dots, n).$
\end{prop}
\begin{proof}

The isomorphism is established by basis permutation:
$$e'_1=e_1, \quad e'_i=e_{\pi^{i-1}(1)}, \ 2 \leq  i \leq n.$$
\end{proof}

\begin{thm}\label{thm35}
Any evolution algebra of permutation $E_{n,\tau}(a_1, a_2, \dots, a_n)$ with $\tau=(l_1, l_2, \dots, l_n)$ and condition $a_i\neq 0$ for all $i \ (1\leq i \leq n)$ is isomorphic to the algebra $E_n.$
\end{thm}

\begin{proof} Taking into account Proposition \ref{prop34} it is sufficient to establish isomorphism between
evolution algebra $E_{n,\pi}(a_1, a_2, \dots, a_n)$ with $\pi=(1, 2, \dots, n)$ and non zero values of $a_i, \ 1\leq i \leq n$ and evolution algebra $E_n.$

The application of the following scaling of basis: $$e'_i = A_ie_i, \ 1\leq i \leq n \quad \mbox{with} \
A_i = \sqrt[2^n-1]{\frac {1}{a_i^{2^{n-1}}a_{i+1}^{2^{n-2}}
\dots a_n^{2^{i-1}}a_1^{2^{i-2}}a_2^{2^{i-3}}\dots a_{i-1}}},$$ deduces products
$$\left\{\begin{array}{ll}e'_i\cdot e'_i = e'_{i+1}, & 1 \leq i \leq n-1,\\[1mm]
e'_n\cdot e'_n = e'_{1}, & \\[1mm]
e'_i\cdot e'_j = 0, & i \neq j.
\end{array}\right.$$
\end{proof}

Now we consider the case of $a_i=0$ for some $i\in\{1, 2, \dots, n\}$.

\begin{prop} \label{prop36}
Any evolution algebra of permutation $E_{n,\pi}(a_1, a_2, \dots,
a_n)$ with $\pi=(l_1, l_2, \dots, l_n)$ and condition $a_i=0$ for
some $i\in\{1, 2, \dots, n\}$ is isomorphic to the algebra $EN_{k_1}\oplus
EN_{k_2}\oplus \dots \oplus EN_{k_r}.$
\end{prop}
\begin{proof}
Similarly to the proof of Proposition \ref{prop34} taking the
change
$$e'_1=e_1, \quad e'_i=e_{\pi^{i-1}(1)}, \ 2 \leq  i \leq n,$$
we can suppose $\pi=(1, 2, \dots, n).$

Let $a_{i_1}=a_{i_2} = \dots =a_{i_r} =0$ for $i_1 < i_2 <
\dots < i_r$ and the rest are non-zero.

If $i_r=n$, then similarly as above we can assume that all $a_i=1$ for $1\leq i \leq n-1$ and hence,
$E_{n,\pi}(a_1, a_2, \dots, a_n)$ with $\pi=(1, 2, \dots, n)$ is isomorphic to the algebra $E_n.$

If $i_r<n$, then taking the following change of basis:
$$e_s^1=e_{i_r+s}, \ 1 \leq s \leq n-i_r,$$
$$e^1_{n-i_r+s}=e_s, \ 1 \leq s \leq i_1,$$
$$e^2_{s} = e_{i_1+s}, \ 1 \leq s \leq i_2-i_1,$$
$$\dots \quad \quad \dots \quad \quad \dots$$
$$e^r_{s} = e_{i_{r-1}+s},\ 1\leq s \leq i_r-i_{r-1}.$$
we obtain that $E_{n,\pi}(a_1, a_2, \dots, a_n)$ with $\pi=(1, 2,
\dots, n)$ is isomorphic to the algebra $EN_{k_1}(a_1, \dots,
a_{k_1-1})\oplus EN_{k_2}(b_1, \dots, b_{k_2-1})\oplus \dots
\oplus EN_{k_r}(c_1, \dots, c_{k_r-1}),$ where each of the algebra
$EN_{k_s}(*, *, \dots, *), \ 1\leq s \leq r$ has the model of the
following evolution algebra:
$$EN_{k}(a_1, \dots, a_{k-1}): \quad \left\{\begin{array}{ll}e_i\cdot e_i = a_ie_{i+1}, & 1 \leq i \leq k-1, \ a_i\neq 0,\\[1mm]
e_i\cdot e_j = 0, & i \neq j.
\end{array}\right.$$
Taking the basis transformation in the algebra $EN_{k}(a_1, \dots, a_{k-1}):$
$$e_1'=e_1, \ e_2' = a_1e_2, \
e_3'=a_1^2a_2e_3, \ \dots, \ e_k' =
a_1^{2^{k-2}}a_2^{2^{k-3}}\dots a_{k-1}e_k$$ we have that
$EN_{k}(a_1, \dots, a_{k-1})$ is isomorphic to the algebra
$EN_{k}$, which complete the proof of proposition.
\end{proof}

Below we establish the isomorphism of evolution algebras of permutations with given conjugated permutations.

\begin{thm}  If permutations $\pi_1, \pi_2 \in  S_n$ are conjugated, then evolution algebras $E_{n,\pi_1}(a_1, \dots, a_{n})$ and $E_{n,\pi_2}(a_1, \dots, a_n)$ are isomorphic.
\end{thm}
\begin{proof} Let $\pi_1, \pi_2 \in  S_n$ are conjugated, then there exists $g \in S_n$ such that $g \pi_1 =  \pi_2g$.
The map $f: E_{n,\pi_1}\rightarrow E_{n,\pi_2}$ defined by $f(e_i) =
e_{g(i)}$ is isomorphism. Indeed,
$$ a_ie_{g\pi_1(i)}= f(a_ie_{\pi_1(i)})=f(e_i \cdot e_i)=f(e_i) \cdot f(e_i) =
a_ie_{g(i)} \ast e_{g(i)} = a_ie_{\pi_2g(i)}. $$
\end{proof}

Thus, for an algebra $E_{n,\pi}(a_1, \dots, a_n)$ we can always assume that $\pi= (1, 2, \dots , n)$ and table of multiplication is $$\left\{\begin{array}{ll}e_i\cdot e_i = a_ie_{i+1}, & 1 \leq i \leq n-1,\\[1mm]
e_n\cdot e_n = a_ne_{1}, & \\[1mm]
e_i\cdot e_j = 0, & \ i \neq j,
\end{array}\right.$$ where $ a_i \in  \{0; 1\}$.

\begin{prop} An arbitrary evolution algebra $E_{n,\pi}(a_1, a_2, \dots , a_n)$ with $\pi= (1, 2, \dots , n)$
is isomorphic to the algebra $E_n$ or the direct sum of evolution
algebras $EN_{k_1}\oplus EN_{k_2}\oplus \dots \oplus EN_{k_r}.$
\end{prop}

\begin{proof} If all $a_i\neq 0$, then due to Theorem \ref{thm35} we have that algebra $E_{n,\pi}(a_1, a_2, \dots , a_n)$ is isomorphic to the $E_n$.

Let $a_k=0$ for some $k$ and $a_i\neq 0$ for $i \neq k.$ Taking the basis transformation in the following form:
$$e_1'=A_1e_{k+1}, \ e_2'=A_2e_{k+2}, \ \dots ,\ e_{n-k}'=A_{n-k}e_n, $$
$$e_{n-k+1}'=A_{n-k+1}e_1, \ e_{n-k+2}'=A_{n-k+2}e_2, \dots,
e_n'=A_ne_k,$$ where
$$A_1=1, \quad A_2 = a_{k+1}, \quad A_3 = a_{k+1}^2a_{k+2}, \quad \dots, \quad  A_{n-k+1} = a_{k+1}^{2^{n-k-1}}
a_{k+2}^{2^{n-k-2}}\dots a_n,$$
$$A_{n-k+2} = a_{k+1}^{2^{n-k}}
a_{k+2}^{2^{n-k-1}}\dots a_n^2a_1, \quad \dots, \quad A_{n} =
a_{k+1}^{2^{n-2}} a_{k+2}^{2^{n-3}}\dots
a_n^{2^{k-1}}a_1^{2^{k-2}}a_2^{2^{k-3}}\dots a_{k-1},$$ we derive
isomorphism between algebra $E_{n,\pi}(a_1, \dots, a_{k-1}, 0,
a_{k+1}, \dots, a_n)$ and algebra $EN_n.$

Applying similar arguments, we can establish that for ($r+1$)-times of parameters
$a_i$ are equal to zero an algebra $E_n(a_1, a_2, \dots, a_n)$ is isomorphic to
$$EN_{k_1}\oplus EN_{k_2}\oplus \dots \oplus EN_{k_r}.$$
\end{proof}
We resume the above results in the main theorem of this section.

\begin{thm} \label{thm39}
An arbitrary evolution algebra of permutations $E_{n,\pi}(a_1, a_2, \dots,
a_n)$ is isomorphic to a direct sum of algebras
$E_{p_1}, \  E_{p_2}, \  \dots, \ E_{p_s}, \ EN_{k_1}, \ EN_{k_2},  \ \dots, \ EN_{k_r},$
i.e.,
$$E_{n,\pi}(a_1, a_2, \dots, a_n) \cong E_{p_1}\oplus E_{p_2} \oplus \dots \oplus  E_{p_s} \oplus EN_{k_1} \oplus EN_{k_2} \oplus \dots \oplus EN_{k_r}.$$
\end{thm}

In the description of evolution algebras of permutations from above theorem we get the importance of algebras $E_n,$ $EN_n$, idempotents and absolute nilpotent elements.

\

\section{Associative enveloping algebras of some evolution algebras}

\

For a complex two-dimensional evolution algebra $E$ of the list of Theorem \ref{thm22} we describe its
associative enveloping algebra $M(E)$:

$M(E_1) = alg\left<R_{e_1}=\left(\begin{matrix} 1& 0\\[1mm] 0&0 \end{matrix}
\right)\right>\cong xx=x;$

$M(E_2) = alg\left<R_{e_1}=\left(\begin{matrix} 1& 0\\[1mm] 0&0 \end{matrix}
\right), \ R_{e_2}=\left(\begin{matrix} 0& 0\\[1mm] 1&0 \end{matrix}
\right)\right>\cong xx=x, \ yx=x;$

$M(E_3) = alg\left<R_{e_1}=\left(\begin{matrix} 1& 1\\[1mm] 0&0 \end{matrix}
\right), \ R_{e_2}=\left(\begin{matrix} 0& 0\\[1mm] -1&-1 \end{matrix}
\right)\right>\cong xx=x, \ xy=-x, \ yx=y, \ yy=-y;$

$M(E_4) = alg\left<R_{e_1}=\left(\begin{matrix} 0& 1\\[1mm] 0&0 \end{matrix}
\right) \right>\cong xx=0;$

$M(E_5(0,0)) = alg\left<R_{e_1}=\left(\begin{matrix} 1& 0\\[1mm] 0&0 \end{matrix}
\right), \ R_{e_2}=\left(\begin{matrix} 0& 0\\[1mm] 0&1 \end{matrix}
\right)\right> \cong xx=x, \ yy=y;$

$M(E_5(a_2,a_3)) (a_2  =0  \ or \ a_3 = 0) \cong \left\{P \in
M_2(\mathbb{C}) \  | \
P= \left(\begin{matrix} b_1& b_2\\[1mm] 0&b_3 \end{matrix}
\right)\right\};$

$M(E_5(a_2,a_3)) (a_2a_3 \neq 0) = alg\left<R_{e_1}=\left(\begin{matrix} 1& a_2\\[1mm] 0&0 \end{matrix}
\right), \ R_{e_2}=\left(\begin{matrix} 0& 0\\[1mm] a_3&1 \end{matrix}
\right)\right> \cong M_2(\mathbb{C});$

$M(E_6(a_4)) = alg\left<R_{e_1}=\left(\begin{matrix} 0& 1\\[1mm] 0&0 \end{matrix}
\right), \ R_{e_2}=\left(\begin{matrix} 0& 0\\[1mm] 1&a_4 \end{matrix}
\right)\right> \cong M_2(\mathbb{C}).$

Take the element $x=\sum\limits_{i=1}^nx_ie_i.$ Let $x \cdot x =0,$
then we have
$$x \cdot x = \big(\sum\limits_{i=1}^nx_ie_i\big) \big(\sum\limits_{i=1}^nx_ie_i\big) = \sum\limits_{i=1}^nx^2_i\sum\limits_{j=1}^na_{i,j}e_j= \sum\limits_{j=1}^n \big(\sum\limits_{i=1}^nx_i^2a_{i,j}\big) e_j.$$

From this we have
$$\begin{cases}a_{1,1}x_1^2 + a_{2,1}x_2^2+\dots+a_{n,1}x_n^2 =0 \\[1mm]
a_{1,2}x_1^2+ a_{2,2}x_2^2+\dots+a_{n,2}x_n^2 =0 \\[1mm]
\dots \dots \dots \dots \dots \dots\dots\dots\dots \\[1mm]a_{1,n}x_1^2 + a_{2,n}x_2^2+\dots+a_{n,n}x_n^2 =0 \\[1mm]\end{cases}$$

This system has a non-trivial solution if and only if
$detA^t=detA=0.$

Note that for the case of Markov evolution algebra, by summing above
equalities, we conclude $x_1^2 + x_2^2+\dots +x_n^2=0.$ If Markov
evolution algebra is real, then $x_i=0, \ 1\leq i \leq n,$ that is,
$x=0.$

Thus we proved the following proposition.
\begin{prop} \label{prop41} A complex evolution algebra has non-trivial absolute nilpotent elements if and only if $detA=0.$ Moreover, if real evolution algebra is Markov, then it has only trivial absolute nilpotent elements.
\end{prop}

For an $n$-dimensional evolution algebra $E$ we consider an
associative enveloping algebra $M(E).$ Due to linearity the right
multiplication operator one gets $M(E)=alg <R_{e_i} \ | \ e_i \ \
\mbox{is \ a \ basis \ element \ of} \ E>$.

Using the equalities $R_{e_i}(e_i)=\sum\limits_{k=1}^n a_{i,k}e_{k}$
and $R_{e_i}(e_j)=0$ for $i\neq j,$ we obtain a matrix form of
$R_{e_i}$ as follows:
$$R_{e_i}=\sum\limits_{k=1}^n
a_{i,k}e_{i,k},$$ where $e_{i,k}$ are matrix units.

One can get
\begin{equation} \label{eq41}
R_{e_i} R_{e_j} =
a_{i,j}\sum\limits_{k=1}^n a_{j,k}e_{i,k}.
\end{equation}

From \eqref{eq41} we conclude that $dim M(E) = r_1+r_2+ \dots+r_n$ with $$r_i=rank\left(\begin{matrix} a_{i,1}a_{1,1}& a_{i,1}a_{1,2} &
\dots& a_{i,1}a_{1,n}
\\[1mm] a_{i,2}a_{2,1}& a_{i,2}a_{2,2} & \dots& a_{i,2}a_{2,n}\\[1mm]
\vdots&\vdots&\vdots&\vdots\\[1mm] a_{i,n}a_{n,1}& a_{i,n}a_{n,2} & \dots& a_{i,n}a_{n,n} \end{matrix}
\right).$$

Further we shall consider some cases for $n$-dimensional evolution algebra $E$ with structural constant matrix $A$ and satisfying the condition: $dim M(E)=n.$

\begin{prop} Let $rank A=1.$ Then associative enveloping algebra $M(E)$
is isomorphic to one of the following algebras:
$$M^s \ : \ x_i x_j
= x_i, \quad 1 \leq i \leq s, \ 1 \leq j \leq n.$$
\end{prop}
\begin{proof} Due to conditions of proposition we have the table of multiplication of algebra $E:$
$$e_i \cdot e_i = t_i\sum\limits_{k=1}^na_{k}e_k.$$

The condition $dim M(E)=n$ implies $t_i \neq 0$ for any $i \ (1
\leq i \leq n)$ and basis of $M(E)$ can be chosen as $\{R_{e_1},
R_{e_2}, \dots, R_{e_n}\},$ where $R_{e_i} =
t_i\sum\limits_{k=1}^na_{k}e_{i,k}$. Therefore, the table of
multiplication of the algebra $M(E)$ has the form: $$R_{e_i}
R_{e_j} = t_ja_jR_{e_i}, \ 1\leq i, j \leq n.$$

By appropriate shifting of basis elements, without loss of generality, one can assume $a_i\neq 0$ for $1\leq i \leq s, \ s\leq n$ and $a_j=0$ for $s+1 \leq j \leq n.$ The scaling the basis elements $e_i, \ 1\leq i \leq s$ reduces our study to the case of $t_ia_i=1.$ Thus, we obtain
the tables of multiplication of associative enveloping algebras $M^s.$
\end{proof}

Below we present the description of $n$-dimensional associative enveloping algebras $M(E)$ for evolution algebra $E$ with $rank A=n.$

\begin{prop} \label{prop43} Let $rank A=n.$ Then associative enveloping algebra $M(E)$ is isomorphic to the algebra
$$M_1: \quad  \ x_i x_i
= x_i, \quad 1 \leq i \leq n.$$
\end{prop}
\begin{proof} Taking into account in the equalities
$$R_{e_i} R_{e_j}= a_{i,j}\sum\limits_{k=1}^n
a_{j,k}e_{i,k} = \beta_1R_{e_1}+\beta_2 R_{e_2} +\dots + \beta_n
R_{e_n}$$
that $\{R_{e_1},R_{e_2},\dots, R_{e_n}\}$ are linear independent (they are forms a basis of $M(E)$),
we conclude $\beta_k=0$ for $k\neq i.$

Therefore,
\begin{equation} \label{eq42} a_{i,j}\sum\limits_{k=1}^n
a_{j,k}e_{i,k} = \beta_iR_{e_i}=\beta_i\sum\limits_{k=1}^n a_{i,k}e_{i,k}.
\end{equation}
From \eqref{eq42} we derive

\begin{equation}\label{eq43} a_{i,j}a_{j,k} = \beta_i a_{i,k}, \quad 1 \leq k \leq n.
\end{equation}

The condition $rank A =n,$ implies $rank\left(\begin{matrix} a_{i,1}&
a_{i,2} & \dots& a_{i,n}
\\[1mm] a_{j,1}& a_{j,2} & \dots& a_{j,n} \end{matrix}
\right)=2$ for any $i\neq j.$

Using the arbitrariness $k$ in the equality \eqref{eq43} we obtain
$$a_{i,j} =0, \quad i\neq j \quad \Rightarrow  \quad a_{i,i}\neq 0.$$

Therefore, $$\ R_{e_i} R_{e_i} = a_{i,i}R_{e_i}, \quad 1 \leq i
\leq n.$$

By scaling the basis elements, we can suppose $a_{i,i} =1$ and the algebra $M_1$ is obtained.
\end{proof}

The list of $n$-dimensional algebras $M(E)$ for an evolution algebra, satisfying the condition $rank A=n-1$, is presented in the following theorem.
\begin{thm}  Let $rank A=n-1.$ Then associative enveloping algebra $M(E)$ is isomorphic to one of the following algebras:
$$M_2: \quad x_ix_i=x_i, \ 1 \leq i \leq n, \  x_1x_n=x_1, \ x_nx_1=x_n,$$
$$M_3: \quad x_ix_i=x_i, \ 1 \leq i \leq n-1, \ x_nx_1=x_n,$$
$$M_4: \quad x_ix_i = x_i, \ 1 \leq i \leq n-1, \quad  x_1 x_2=x_n, \quad x_1 x_n=x_n,
\quad x_n x_2=x_n.$$
\end{thm}
\begin{proof}
Without loss of generality, one
can assume $$A=\left(\begin{matrix} a_{1,1}& a_{1,2} & \dots&
a_{1,n}
\\[1mm] a_{2,1}& a_{2,2} & \dots& a_{2,n}\\[1mm]
\vdots&\vdots&\vdots&\vdots\\[1mm] a_{n-1,1}& a_{n-1,2} & \dots& a_{n-1,n}
\\[1mm] \sum\limits_{i=1}^{n-1}\alpha_ia_{i,1}& \sum\limits_{i=1}^{n-1}\alpha_ia_{i,2}
& \dots& \sum\limits_{i=1}^{n-1}\alpha_ia_{i,n} \end{matrix}
\right).$$

\textbf{Case 1.} Let
$R_{e_n}=\sum\limits_{k=1}^{n}\big(\sum\limits_{i=1}^{n-1}\alpha_ia_{i,k}
\big)e_{n,k} \neq 0.$ Then $\{R_{e_1}, R_{e_1}, \dots,
R_{e_{n-1}},R_{e_{n}} \}$ is a basis of $M(E).$

Similarly as in the proof of Proposition \ref{prop43}, from the condition
$$rank\left(\begin{matrix} a_{i,1}& a_{i,2} & \dots& a_{i,n}
\\[1mm] a_{j,1}& a_{j,2} & \dots& a_{j,n} \end{matrix}
\right)=2,$$ for any $1 \leq i,j \leq n-1, \ i\neq j,$ we deduce
$$a_{i,j} =0, \quad 1 \leq i,j\leq n-1, \ i\neq j.$$

Consequently,
$$R_{e_i} = a_{i,i}e_{i,i} + a_{i,n}e_{i,n}, \ 1 \leq i
\leq n-1$$ and $$ R_{e_n} = \sum\limits_{k=1}^{n-1}
\alpha_ka_{k,k}e_{n,k} + (\sum\limits_{s=1}^{n-1}
\alpha_sa_{s,n})e_{n,n}.$$

If $\alpha_ka_{k,k} \neq 0$ for some $k  \ (1\leq k \leq n-1)$, then
without loss of generality we can suppose $\alpha_1a_{1,1} \neq
0.$

Consider
$$R_{e_n} R_{e_1} =\alpha_1a_{1,1}^2e_{n,1} +
\alpha_1a_{1,1}a_{1,n}e_{n,n}.$$ Since the product $R_{e_n} R_{e_i}$
should be expressed by $R_{e_n}$, we conclude $\alpha_ka_{k,k} =
\alpha_ka_{k,n} = 0, \ 2 \leq k \leq n-1,$ which yield
$$\alpha_k =0, \quad 2 \leq k \leq n-1.$$

If $\alpha_ka_{k,k} = 0$ for any $k \ (1\leq k \leq n-1),$ then from
the condition $R_{e_n} \neq 0$ we have the existence some $k_0$ such
that $\alpha_{k_0} \neq 0$. Hence, $a_{k_0,k_0} = 0.$
Since $rank A =n-1,$ then $a_{k,k} \neq 0$ for any $k
\neq k_0,$ consequently, $\alpha_k =0$ for any $k \neq k_0.$
Without loss of generality, one can assume $k_0=1.$

Thus, in Case 1 we obtain
$$\alpha_k =0, \ 2 \leq k \leq n-1.$$

Therefore, $$R_{e_n}  R_{e_i} = 0, \ 2 \leq i \leq n-1, \quad
\quad R_{e_1}  R_{e_n} =a_{1,n}R_{e_1}.$$

Consider
$$R_{e_i}  R_{e_n} = (a_{i,i}e_{i,i} +
a_{i,n}e_{i,n})  (\alpha_1a_{1,1}e_{n,1} +
\alpha_1a_{1,n}e_{n,n})= a_{i,n}(\alpha_1a_{1,1}e_{i,1} +
\alpha_1a_{i,n}e_{i,n}).$$

Since $R_{e_i} R_{e_n}$ should belong to $<R_{e_i}>$, we get
$a_{i,n} =0, \quad 2 \leq i \leq n-1.$

Thus, the table of multiplication of the algebra $M(E)$ has the
following form:
$$R_{e_i}  R_{e_i} = a_{i,i}R_{e_i} \quad 1 \leq i \leq n-1,$$
$$R_{e_1} R_{e_n} = a_{1,n}R_{e_1},  \quad R_{e_n}  R_{e_1} = a_{1,1}R_{e_n},
\quad R_{e_n}  R_{e_n} = a_{1,n}R_{e_n},$$ where $(a_{1,1},
a_{1,n}) \neq (0,0)$ and $a_{i,i} \neq 0,$ $2 \leq i \leq n.$

Considering the possible cases: $a_{1,1}a_{1,n} \neq 0$ and
$a_{1,1}a_{1,n}=0$, one finds the algebras $M_2, \ M_3.$

\textbf{Case 2.} Let $R_{e_n}= 0.$ Since a non-zero products of the form $R_{e_i} R_{e_j} = a_{i,j}\sum\limits_{k=1}^n a_{j,k}e_{i,k}$ are
linear independent, we obtain the existence of a unique  non-zero
coefficient $a_{i_0,j_0}$, $1 \leq i_0, j_0\leq n-1,$ $i_0\neq
j_0$. Without loss of generality, we can suppose $i_0=1, j_0=2$,
i.e. $a_{1,2} \neq 0$.

Therefore,
$$a_{i,j}=0, \quad 1 \leq i, j\leq n-1,\quad i\neq j, \quad (i,j) \neq (1,2).$$
$$a_{i,i}\neq 0, \quad 1 \leq i \leq n-1,$$

Putting $x_i = R_{e_i}, 1 \leq i \leq n-1$ and $x_n =
a_{2,2}e_{1,2} + a_{2,n}e_{1,n},$  we obtain the table of
multiplications of the algebra $M(E)$ in the form:
$$x_ix_i = a_{i,i}x_i, \ 1 \leq i \leq n-1, \quad  x_1 x_2=a_{1,2}x_n, \quad x_1 x_n=a_{1,1}x_n,
\quad x_n x_2=a_{2,2}x_n.$$

Taking the change $$x_1' = \frac {1} {a_{i,i}}x_i, \ 1 \leq i \leq
n-1, \quad x_n'=\frac {a_{1,2}}{a_{1,1}a_{2,2}} x_n$$ we get the algebra $M_4.$
\end{proof}

\

\section{Three-dimensional evolution algebras whose generators have infinite period.}

\

In this section we study three-dimensional evolution algebras whose generators have infinite period.

Let $E$ be a three-dimensional evolution algebra $E$ with
table of multiplications:
\begin{equation} \label{eq51} e_1 \cdot e_1 = a_1e_1+a_2e_2+a_3e_3,\quad e_2
\cdot e_2 = b_1e_1+b_2e_2+b_3e_3, \quad e_3 \cdot e_3 =
c_1e_1+c_2e_2+c_3e_3.
\end{equation}
Since the periods of generators are infinite, we get $a_1=b_2 =c_3=0.$

Consider $e_i^{[3]}$ and $e_i^{[4]}$ for $1\leq i \leq 3$
$$\begin{cases}e_1^{[3]} = (a_2^2b_1+a_3^2c_1)e_1+a_3^2c_2e_2 +
a_2^2b_3e_3,\\[1mm] e_2^{[3]} = b_3^2c_1e_1 +
(b_1^2a_2+b_3^2c_2)e_2+b_1^2a_3e_3,\\[1mm]
e_3^{[3]} = c_2^2b_1e_1+c_1^2a_2e_2 +
(c_1^2a_3+c_2^2b_3)e_3,\end{cases}$$
$$\begin{cases}e_1^{[4]} = (a_3^4c_2^2b_1 + a_2^4b_3^2c_1)e_1 + a_2^4b_3^2c_2e_2 +
a_3^4c_2^2b_3e_3,\\[1mm] e_2^{[4]} = b_1^4a_3^2c_1e_1 +
(b_3^4c_1^2a_2+b_1^4a_3^2c_2)e_2+b_3^4c_1^2a_3e_3,\\[1mm] e_3^{[4]} =
c_1^4a_2^2b_1e_1 + c_2^4b_1^2a_2e_2 +
(c_2^4b_1^2a_3+c_1^4a_2^2b_3)e_3.\end{cases}$$

Taking account the condition on periods of generators, we derive
\begin{equation} \label{eq52}
a_2^2b_1+a_3^2c_1=0, \quad b_1^2a_2+b_3^2c_2=0, \quad c_1^2a_3+c_2^2b_3=0,
\end{equation}
\begin{equation} \label{eq53}
a_3^4c_2^2b_1 + a_2^4b_3^2c_1=0, \quad b_3^4c_1^2a_2+b_1^4a_3^2c_2=0, \quad c_2^4b_1^2a_3+c_1^4a_2^2b_3=0.
\end{equation}

\begin{thm} Let $E$ be a three-dimensional evolution algebra with the table of multiplication \eqref{eq51}. Let any basis element has an infinite period and $a_2a_3b_1b_3c_1c_2=0.$ Then $E$ is isomorphic to the
following evolution algebra:
$$E^1: \ e_1 \cdot e_1 = a_2e_2+a_3e_3,\quad e_2
\cdot e_2 = b_3e_3.$$
\end{thm}
\begin{proof}
Let $a_2a_3b_1b_3c_1c_2=0,$ then, without loss of generality, we can
assume $b_1=0.$ The equalities \eqref{eq52} and \eqref{eq53} imply

\begin{equation} \label{eq54}
a_3c_1=0, \quad b_3c_2=0, \quad a_2b_3c_1=0.
\end{equation}

\textbf{Case 1.} Let $a_3=b_3=0.$ Then we obtain products
$$e_1 \cdot e_1 = a_2e_2,\quad e_3
\cdot e_3 = c_1e_1+c_2e_2.$$

Taking the change $e_1'=e_3, e_2'=e_1, e_3'=e_2,$ we derive that
this algebra is isomorphic to the algebra $E^1.$

\textbf{Case 2.} Let $a_3=0$ and $b_3\neq 0.$ Then from \eqref{eq54} we have $c_2=a_2c_1=0.$

If $a_2 \neq 0,$ then $c_1=0$ and we obtain evolution algebra with multiplications:
$$e_1
\cdot e_1 = a_2e_2, \ e_2 \cdot e_2 = b_3e_3.$$

If $a_2=0,$ then the table of multiplications of the algebra $E$ is as follows:
$$e_2 \cdot e_2 = b_3e_3, \ e_3
\cdot e_3 = c_1e_1.$$

Putting $e_1'=e_2, e_2'=e_3, e_3'=e_1,$ we derive that this algebra is isomorphic to the algebra
$E^1.$

\textbf{Case 3.} Let $a_3\neq 0.$ Then restrictions \eqref{eq54} imply $c_1=b_3c_2=0.$

If $b_3 \neq 0,$ then $c_2=0$ and the algebra $E^1$ is obtained.

If $b_3=0,$ then by taking basis transformation as follows:
$$e_1'=e_1, e_2'=e_3, e_3'=e_2$$ we get $E^1.$
\end{proof}

\begin{thm} Let $E$ be a three-dimensional evolution algebra with the table of multiplication \eqref{eq51} and let $a_2a_3b_1b_3c_1c_2\neq 0.$ Then period of each basis elements of the algebra $E$ is infinite if and only if
equations \eqref{eq52} hold true.
\end{thm}
\begin{proof} It is sufficient to proof the part \emph{Only if.} From equalities \eqref{eq52}
we get $$a_2 = -\frac {b_3^2c_2} {b_1^2}, \quad a_3 = -\frac
{c_2^2b_3} {c_1^2}.$$

Putting this restrictions to the equality
$a_2^2b_1+a_3^2c_1=0$ we obtain
$$0=a_2^2b_1+a_3^2c_1 = \frac{b_3^2c_2^2
(b_3^2c_1^3+b_1^3c_2^2)} {b_1^3c_1^3},$$ which implies
\begin{equation}\label{eq551}b_3^2c_1^3+b_1^3c_2^2=0. \end{equation}

Similarly, we derive

\begin{equation}\label{eq55}b_1 = -\frac{a_3^2c_1} {a_2^2}, \quad b_3 = -\frac{c_1^2a_3}
{c_2^2}  \quad  \Rightarrow \quad a_3^2c_2^3+a_2^3c_1^2=0,
\end{equation}

\begin{equation}\label{eq56}c_1= -\frac{a_2^2b_1} {a_3^2}, \quad c_2=
-\frac{b_1^2a_2}{b_3^2} \quad \Rightarrow \quad
a_2^2b_3^3+a_3^3b_1^2=0. \end{equation}

Applying induction we will prove the following:
\begin{equation}\label{eq57}e_1^{[k]} = A_{k,2}e_2+ A_{k,3}e_3, \quad e_2^{[k]} = B_{k,1}e_1+ B_{k,3}e_3,  \quad
e_3^{[k]} = C_{k,1}e_1+ C_{k,2}e_2, \  k \geq 3 \end{equation}
with recurrence expressions
\begin{equation}\label{eq58} \begin{array}{lll}A_{k,2} = A^2_{k-1,3}c_2, & A_{k,3} = A^2_{k-1,2}b_3,
& A_{k-1,2}^2b_1+A_{k-1,3}^2c_1=0,\\[1mm]
B_{k,1} = B^2_{k-1,3}c_1, & B_{k,3} = B^2_{k-1,1}a_3, &
B_{k-1,1}^2a_2+B_{k-1,3}^2c_2=0,\\[1mm]
C_{k,1} = C^2_{k-1,2}b_1, & C_{k,2} = C^2_{k-1,1}a_2, &
C_{k-1,1}^2a_3+C_{k-1,3}^2b_3=0,
\end{array}
\end{equation}
where $A_{2,2} = a_2, \ A_{2,3} = a_3, \ B_{2,1} =b_1, \ B_{2,3}
=b_3, \ C_{2,1}=c_1, \ C_{2,2} =c_2.$

In fact, the correctness of expressions \eqref{eq57} is equivalent
to that each basis element of evolution algebra $E$ has infinite
period.

From decompositions of $e_i^{[3]}, \ 1\leq i \leq 3$ it is easy to
see the correctness of \eqref{eq58} for $k=3,$ i.e. $$A_{3,2} =
a_3^2c_2 = A^2_{2,3}c_2 , \quad A_{3,3} = a_2^2b_3 = A^2_{2,2}b_3,
\quad A_{2,2}^2b_1+A_{2,3}^2c_1=a_2^2b_1+a_3^2c_1=0,$$
$$B_{3,1} = b_3^2c_1= B^2_{2,3}c_1, \quad B_{3,3} = b_1^2a_3 = B^2_{2,1}a_3,
\quad B_{2,1}^2a_2+B_{2,3}^2c_2=b_1^2a_2+b_3^2c_2=0,$$
$$C_{3,1} =c_2^2b_1= C^2_{2,2}b_1, \quad C_{3,2} = c_1^2a_2= C^2_{2,1}a_2, \quad C_{2,1}^2a_3+C_{2,2}^2b_3=c_1^2a_3+c_2^2b_3=0.$$

Suppose that \eqref{eq57} and \eqref{eq58} are true for $k.$ We
will prove it for $k+1.$

The chain of equalities
$$e_1^{[k+1]} = e_1^{[k]} \cdot e_1^{[k]} =
(A_{k,2}e_2 + A_{k,3}e_3)\cdot (A_{k,2}e_2 +
A_{k,3}e_3)=(A_{k,2}^2b_1+ A_{k,3}^2c_1)e_1 +A_{k,3}^2c_2e_2 +
A_{k,2}^2b_3e_3$$ deduces
$$A_{k+1,2} = A_{k,3}^2c_2, \quad A_{k+1,3} = A_{k,2}^2b_3.$$

Applying induction assumption, that is,
$$A_{k,2} = A_{k-1,3}^2c_2, \quad A_{k,3} = A_{k-1,2}^2b_3 , \quad \quad A_{k-1,2}^2b_1+ A_{k-1,3}^2c_1=0,$$
we get
$$A_{k,2}^2b_1+ A_{k,3}^2c_1= A_{k-1,3}^4c_2^2b_1 + A_{k-1,2}^4b_3^2c_1= \frac {A_{k-1,2}^4b_1^2} {c_1^2}c_2^2b_1 +
 A_{k-1,2}^4b_3^2c_1 =  \frac {A_{k-1,2}^4} {c_1^2}(b_1^3c_2^2 + b_3^2c_1^3)$$

The equality \eqref{eq551} implies $A_{k,2}^2b_1+ A_{k,3}^2c_1=0.$

Similarly, one finds


$$B_{k,1}^2a_2 + B_{k,3}^2c_2=0, \quad \quad C_{k,1}^2a_3+C_{k,2}^2b_3=0.$$
\end{proof}

\section*{ Acknowledgements}
This works is supported by the Grant No.0251/GF3 of Education and Science Ministry of Republic of Kazakhstan and
the Grant (RGA) No:11-018 RG/Math/AS\_I--UNESCO FR: 3240262715.

\end{document}